\documentclass[12pt, doublespace]{amsart}
\usepackage{amssymb}

\usepackage[english]{babel}
\usepackage[autostyle]{csquotes}

\theoremstyle{plain}
\newtheorem{theorem}{Theorem}[section]
\newtheorem{corollary}[theorem]{Corollary}
\newtheorem{lemma}[theorem]{Lemma}
\newtheorem{proposition}[theorem]{Proposition}

\theoremstyle{definition}

\newtheorem{remark}[theorem]{Remark}

\theoremstyle{remark}
\newtheorem*{example}{Example}

\newcommand{\abs}[1]{\lvert#1\rvert}
\newcommand{\norm}[1]{\lVert#1\rVert}

\newcommand{\inner}[2]{\langle #1\rangle #2}

\newcommand{\C}{\mathbb{C}}
\newcommand{\N}{\mathbb{N}}
\newcommand{\B}{\mathcal{B}}
\newcommand{\veps}{\varepsilon}
\newcommand{\la}{\lambda}
\newcommand{\tow}{\stackrel{w^*}{\longrightarrow}}
\newcommand{\wto}{\stackrel{w}{\longrightarrow}}

\DeclareMathOperator{\Range}{Range}

\begin{document}
\baselineskip 18pt
\title{On quasinilpotent operators and the invariant subspace problem}

\author[A.~Tcaciuc]{Adi Tcaciuc}
\address[A. Tcaciuc]{Mathematics and Statistics Department,
   Grant MacEwan University, Edmonton, Alberta, Canada T5J
   P2P, Canada}
\email{tcaciuca@macewan.ca}


\begin{abstract}
We  show  that a  bounded quasinilpotent operator $T$  acting  on  an  infinite
dimensional Banach space has an invariant subspace if and only if there exists a rank one operator $F$ and a scalar $\alpha\in\mathbb{C}$, $\alpha\neq 0$, $\alpha\neq 1$, such that $T+F$ and $T+\alpha F$ are also quasinilpotent. We also prove that for any fixed rank-one operator $F$, almost all perturbations $T+\alpha F$ have invariant subspaces of infinite dimension and codimension.
\end{abstract}

\maketitle


\section{Introduction}\label{intro}

One of the most important unsolved problem in Operator Theory is the \emph{Invariant Subspace Problem}: Does every bounded operator on an infinite  dimensional, separable, complex Hilbert space have a non-trivial invariant closed subspace?  Von Neumann proved the existence of such subspaces for compact operators acting on a separable Hilbert space, a result which was extended by Aronszajn and Smith \cite{AS54} to separable Banach spaces. Lomonosov \cite{L73} greatly increases the class of operators with invariant subspaces by showing that every operator commuting with a compact operator has an invariant subspace. Enflo (\cite{E76, E87}, see also \cite{B85}) constructed the first
example of a bounded operator on a (non-reflexive) Banach space which has no non-trivial invariant subspaces, followed by a construction by Read \cite{R84}. Later Read produced several such examples: strictly singular operators, quasinilpotent operators, and operators acting on $l_1$ (see \cite{R85}, \cite{R97},\cite{R91}). All these examples are on non-reflexive Banach spaces, and the Invariant Subspace Problem is still open for general reflexive Banach spaces. For an overview of the Invariant Subspace Problem see the monographs by Radjavi and Rosenthal \cite{RR03} or the more recent book by Chalendar and Partington \cite{CP11}.

A very important special case for which the Invariant Subspace Problem is still open is that of \emph{quasinilpotent} operators on Hilbert spaces, or, more generally, on reflexive Banach spaces. An operator $T$ is called \emph{quasinilpotent} if $\sigma(T)=\{0\}$, where by $\sigma(T)$ we denote the spectrum of $T$. Substantial work has been devoted over the years to ISP for quasinilpotent operators, in particular on Hilbert spaces. We mention several important papers, without attempting to provide an exhaustive list: Apostol and Voiculescu \cite{AV74}, Herrero \cite{H78}, Foia\c{s} and Pearcy \cite{FP74}, Foia\c{s}, Jung, Ko, and Pearcy,\cite{JKP03, FJKP04, FJKP05}.

In Section \ref{mainsection} we develop a method of investigating invariant subspaces for quasinilpotent operators on complex Banach spaces by examining the resolvent function. In our main result in this section, Theorem \ref{main}, we prove a necessary and sufficient condition for a quasinilpotent operator to have invariant subspaces, a condition which is related to the stability of the spectrum under rank-one perturbations.

 Next we examine the existence of invariant \emph{half-spaces} for rank-one perturbations of quasinilpotent operators. By a \emph{half-space} we understand a closed subspace which is both infinite dimensional and infinite codimensional. A method of examining invariant half-spaces for finite rank perturbations  was introduced by Androulakis, Popov, Tcaciuc, and Troitsky in \cite{APTT09}, where the authors showed that certain classes of bounded operators have rank-one perturbations which admit invariant half-spaces. In \cite{PT13} Popov and Tcaciuc showed that \emph{every} bounded operator $T$ acting on a reflexive Banach space can be perturbed by a rank-one operator $F$ such that $T+F$ has an invariant half-space. Moreover, when a certain spectral condition is satisfied, $F$ can be chosen to have arbitrarily small norm. Recently these results were extended to general Banach spaces in \cite{T17}. In this line of investigation, Jung, Ko, and Pearcy \cite{JKP17, JKP18} adapted this theory to operators on Hilbert spaces, where the presence of additional structure and specific Hilbert space methods allowed them to prove important results regarding the matricial structure of arbitrary operators on Hilbert spaces. For algebras of operators this type of problems have been studied in \cite{P10}, \cite{MPR13}, and \cite{SW16}.  More control on the construction of rank-one perturbation that have invariant half-spaces was achieved in \cite{TW17}. In that paper the authors showed that for  any bounded operators $T$ with countable spectrum  acting on a Banach space $X$, and for any non-zero $x\in X$, one can find a rank one operator with range span$\{x\}$ such that $T+F$ has an invariant subspace.

 In Section \ref{perturb}, we refine the method developed in the previous section to show that \emph{almost all} (in a sense that is made precise in Theorem \ref{main2} below ) rank-one perturbations of quasinilpotent operators have invariant half-spaces.

\section{Invariant subspaces for quasinilpotent operators}\label{mainsection}
For a Banach space $X$, we denote by $\mathcal{B}(X)$ the algebra of all (bounded linear) operators on $X$. When  $T\in{\mathcal B}(X)$, we write $\sigma(T)$, $\sigma_p(T)$,$\sigma_{ess}(T)$, and $\rho(T)$  for the spectrum
of~$T$,  point spectrum of $T$, the essential point spectrum of $T$, and the resolvent set of~$T$, respectively. The closed span of a set  $\{x_n\}_n$ of vectors in $X$ is denoted by $[x_n]$.

 For $T\in\mathcal{B}(X)$, the \emph{resolvent} of $T$ is the function $R:\rho(T)\to\mathcal{B}(X)$ defined by $R(z)=(zI-T)^{-1}$. When $\abs{z}>r(T)$, where $r(T)$ is the spectral radius of $T$, the resolvent is given by the Neumann series expansion
  $$
  R(z)=(zI-T)^{-1}=\sum_{i=0}^{\infty}\frac{T^i}{z^{i+1}}.
  $$
  In particular, when $T$ is quasinilpotent this expansion holds for all complex numbers  $z\neq 0$. The resolvent $R$ is analytic on $\rho(T)$, hence on $\mathbb{C}\setminus\{0\}$ when $T$ is quasinilpotent.

We first prove a  simple lemma which gives sufficient and necessary conditions for $\lambda\in\rho(T)$ to be an eigenvalue for some fixed rank-one perturbation.

\begin{lemma} \label{eigen}
  Let $X$ be a separable Banach space, $T\in\mathcal{B}(X)$, and $F:=e^*\otimes f$ a rank one operator. Fix $\lambda\in\rho(T)$  and $\alpha\in\mathbb{C}\setminus\{0\}$. Then the following are equivalent:
 \begin{enumerate}
 \item $e^*((R(\la)f)=\alpha^{-1}$.
 \item $\la\in\sigma_p(T+\alpha F)$.
 \end{enumerate}
\end{lemma}

\begin{proof}

i)$\Rightarrow$ ii)
We are going to show that $y:=R(\la)f$ is an eigenvector for  $T+\alpha F$, corresponding to the eigenvalue $\la$. Note that $Ty=\lambda y -f$. Then:
$$
(T+\alpha F)y=Ty+\alpha e^*(y)f=\lambda y -f +\alpha e^*(y)f=\lambda y -f +\alpha \alpha^{-1} f=\lambda y.
$$

ii)$\Rightarrow$ i)

Let $y\in X$ be an eigenvector for $T+\alpha F$  corresponding to $\lambda$. Hence $Ty+\alpha e^*(y)f=\lambda y$. Note that since $\lambda\in\rho(T)$, it follows that $e^*(y)\neq 0$. We have:
$$
 Ty+\alpha e^*(y)f=\lambda y \Leftrightarrow (\lambda I-T)y=\alpha e^*(y)f \Leftrightarrow y=\alpha e^*(y)(\lambda I-T)^{-1}f.
$$

Applying $e^*$ to both sides of the last equality, we get that
$$
e^*(y)=\alpha e^*(y)e^*((\lambda I-T)^{-1}f),
$$
 and since $e^*(y)\neq 0$ it follows that $e^*(R(\la)f)=\alpha^{-1}$.

\end{proof}

\begin{remark}\label{remark1}
Note that when $e^*(R(\la)f)=\alpha^{-1}$, from the proof of the previous lemma it follows that $R(\la)f$ is an eigenvector for $T+\alpha F$ corresponding to the eigenvalue $\lambda$.
\end{remark}

We are now ready to prove the main theorem of this section.

\begin{theorem}\label{main}
Let $X$ be a separable Banach space and $T\in\mathcal{B}(X)$ a quasinilpotent operator. Then the following are equivalent:
\begin{enumerate}
  \item $T$ has an invariant subspace.
  \item There exists a rank one operator $F$ such that for any $\alpha\in\mathbb{C}$, $T+\alpha F$ is quasinilpotent.
  \item There exists a rank one operator $F$ and $\alpha\in\mathbb{C}$, $\alpha\neq 0$, $\alpha\neq 1$, such that $T+F$  and $T+\alpha F$ are quasinilpotent.
\end{enumerate}
\end{theorem}

\begin{proof}

 Note first that since $\sigma_{ess}(T)=\{0\}$, and the essential spectrum is stable under compact perturbations, it follows that for any $\alpha\in\mathbb{C}$ and any rank-one operator $F$, $\sigma_{ess}(T+\alpha F)=\{0\}$. Therefore $\sigma(T+\alpha F)$ is at most countable with $0$ the only accumulation point, and any $\lambda\in\sigma(T+\alpha F)\setminus\{0\}$ is an eigenvalue((see e.g. \cite{AA02}, Corollary 7.49 and 7.50). Hence, the condition that $T+\alpha F$ is quasinilpotent, is equivalent to $\sigma_p(T+\alpha F)\setminus \{0\}=\emptyset$.

i) $\Rightarrow$ ii)

Suppose $Y$ is a non-trivial invariant subspace for $T$. Pick $f\in Y$, and $e^*\in X^*$ such that $e^*(Y)=0$. Let $F$ be the rank one operator defined by $F:=e^*\otimes f$. Then, since $Y$ is $T$-invariant and $f\in Y$, we have that the orbit $(T^nf)$ is contained in $Y$, hence for all $n\in\mathbb{N}$, $e^*(T^nf)=0$. It follows that, for any $z\in\C\setminus\{0\}$ we have:
\begin{equation}\label{zero}
  e^*(R(z)f)=e^*\left(\sum_{i=0}^{\infty}\frac{T^i f}{z^{i+1}}\right)=\sum_{i=0}^{\infty}\frac{e^*(T^i f)}{z^{i+1}}=0.
\end{equation}

Fix $\alpha\neq 0$, arbitrary. From (\ref{zero})  and Lemma \ref{eigen} it now follows that for any $z\in\C\setminus\{0\}$ we have that $z\notin\sigma_p(T+\alpha F)$. Therefore, for any $\alpha\neq 0$ we have that $\sigma_p(T+\alpha F)\setminus \{0\}=\emptyset$, hence $T+\alpha F$ is quasinilpotent.

ii) $\Rightarrow$ iii) obvious

iii) $\Rightarrow$ i) We argue by contradiction. Assume that $T$ has no invariant subspaces, and fix $F:=e^*\otimes f$ an arbitrary rank one operator. Since $T$ has no invariant subspaces it follows that $e^*(T^nf)\neq 0$ for infinitely many values of $n$. Indeed, otherwise there exist $k\in\mathbb{N}$ such that $e^*(T^jf)=0$ for all $j\geq k$. However this means that the closed span of $(T^jf)_{j\geq k}$ is contained in the kernel of $e^*$, thus it would be a non-trivial $T$-invariant subspace, contradicting the assumption.

To simplify the notation, we denote by $g:\C\setminus\{0\}\to\C$ the analytic function defined by $g(z)=e^*(R(z)f)$. Note that $g$ has an isolated singularity at $z=0$ and its Laurent series about $z=0$ is

 $$
    g(z)=\sum_{i=0}^{\infty}\frac{1}{z^{i+1}}e^*(T^if).
 $$

Since $e^*(T^nf)\neq 0$ for infinitely many values of $n$, it follows that the Laurent expansion of $g$ will have infinitely many non-zero terms of the form $\frac{1}{z^{i+1}}e^*(T^if)$. Therefore $z=0$ is an isolated \emph{essential} singularity for $g$. From Picard's Great Theorem it follows that $g$ attains any value, with possibly one exception, infinitely often, in any neighbourhood of $z=0$.  Hence, for all $\alpha\neq 0$, with possibly one exception, the set $\{z\in\mathbb{C} :  g(z)=\alpha^{-1}\}$ is infinite. Note that this set is in fact \emph{countably} infinite, as it is easy to see that in Picard's Theorem the values can be attained at most countably many times. Therefore $\sigma_p(T+\alpha F)\setminus\{0\}=\{z\in\mathbb{C} :  g(z)=\alpha^{-1}\}$ is countably infinite for all $\alpha$, with possibly one exception, so $T+\alpha F$ is quasinilpotent for at most one non-zero value $\alpha$. Since $F$ was arbitrary, this contradicts iii), and the implication is proved.
\end{proof}

The techniques employed in the proof of the previous theorem also gives the following characterization of the spectrum of rank-one perturbation of quasinilpotent operators.

\begin{proposition} \label{3options}
Let $X$ be a Banach space, $T\in\mathcal{B}(X)$ a quasinilpotent operator, and $F:=e^*\otimes f$ a rank one operator. Then exactly one of the following three possibilities holds:
\begin{enumerate}
  \item For all $\alpha\in\mathbb{C}$, $T+\alpha F$ is quasinilpotent.
  \item For all non-zero $\alpha\in\mathbb{C}$, with possibly one exception, $\sigma_p(T+\alpha F)$ is countably infinite.
  \item There exists $K\in\N$  such that for all non-zero $\alpha\in\mathbb{C}$, $0<\abs{\sigma_p(T+\alpha F)\setminus\{0\}}<K$.
\end{enumerate}
\end{proposition}

\begin{proof}
From the proof of Theorem \ref{main} the options $(i)$ and $(ii)$ hold when  $e^*(T^nf)=0$ for all $n$, and when $e^*(T^nf)\neq 0$  for infinitely many values of $n$, respectively. It remains to examine the case when $e^*(T^nf)\neq 0$ for finitely, non-zero, values of $n$. Let $k>0$ be the smallest natural number such that $e^*(T^kf)\neq 0$ and $e^*(T^jf)= 0$ for all $j>k$. With the notations from Theorem \ref{main} it follows that:
$$
g(z)= e^*(R(z)f)=\sum_{i=0}^{k}\frac{1}{z^{i+1}}e^*(T^if).
$$
Therefore $z=0$ is a pole of order $k+1$ for $g$. In this case it is easy to see that for any $\alpha\neq 0$, the equation $g(z)=\alpha^{-1}$ has at most $k+1$ solutions,  hence the cardinality of the non-empty set $\sigma_p(T+\alpha F)\setminus\{0\}$ is at most $k+1$, and $(iii)$ holds.

\end{proof}

We will show in the next example  that the second option in Proposition \ref{3options} can indeed hold, and that the one exception is in general unavoidable.
\begin{example} \label{exa}
Let $H$ be a separable Hilbert space, denote by $(e_n)_n$ an orthonormal basis, and define $T\in\mathcal{B}(H)$ to be the weighted shift defined by
$$
Te_n=\frac{1}{n}e_{n+1},\ \text{ for } n=1,2,\dots
$$
It is easy to see that $T$ is a compact quasinilpotent operator. Consider the rank one operator $F\in\mathcal{B}(H)$ defined by $F(x):=\inner{x,f}e_1$, where $f=\sum_{n=1}^{\infty}\frac{1}{n}e_n$. We are going to show that $T-F$ is quasinilpotent, and that for any $\alpha\neq -1$ we have $\sigma_p(T+\alpha F)$ is countably infinite. From the previous considerations this is equivalent to showing that the function $g(z):=\inner{R(z)e_1, f}$ is analytic on $\C\setminus\{0\}$,  has an essential
singularity at $0$, and $g(z)\neq -1$ for all $z\in\C\setminus\{0\}$. We have
$$
R(z)e_1=\sum_{n=0}^{\infty}\frac{1}{z^{n+1}}T^ne_1=\sum_{i=0}^{\infty}\frac{1}{z^{n+1}}\frac{1}{n!}e_{i+1}.
$$
Therefore
$$
g(z):=\inner{R(z)e_1, f}=\inner{\sum_{n=0}^{\infty}\frac{1}{z^{n+1}}\frac{1}{n!}e_{n+1},\sum_{n=1}^{\infty}\frac{1}{n}e_n}=
\sum_{n=1}^{\infty}\frac{1}{n!}\frac{1}{z^n}=\exp(1/z)-1.
$$
Clearly $g$ has an essential singularity at $z=0$ and $g(z)\neq -1$ for any $z\in\C\setminus\{0\}$.
\end{example}

\section{Invariant half-spaces for rank one perturbations}\label{perturb}

We next turn our attention to the study of invariant half-spaces of rank-one perturbations of quasinilpotent operators.  First recall some standard notations and definitions. A sequence $(x_n)_{n=1}^{\infty}$ in $X$ is called a \emph{basic sequence} if any $x\in[x_n]$ can be written uniquely as $x=\sum_{n=1}^{\infty} a_n x_n$, where the convergence is in norm (see \cite[section 1.a]{LT77} for background on Schauder bases and basic sequences). As $[x_{2n}]\cap[x_{2n+1}]=\{0\}$ it is immediate that $[x_{2n}]$ is of both infinite dimension and infinite codimension in $[x_n]$, thus a half-space, and since every Banach space contains a basic sequence, it follows that every infinite dimensional Banach space contains a half-space.

An important tool that we are going to use is the following criterion of Kadets and Pe{\l}czy\'{n}ski for a subset of Banach space to contain a basic sequence (see, e.g., \cite[Theorem 1.5.6]{AK06})

 \begin{theorem}[Kadets, Pe{\l}czy\'{n}ski] \label{criterion}
 Let $S$ be a bounded subset of a Banach space $X$ such that $0$ does not belong to the norm closure of $S$. Then the following
 are equivalent:
 \begin{enumerate}
 \item $S$ fails to contain a basic sequence,
 \item The weak closure of $S$ is weakly compact and fails to contain $0$.
 \end{enumerate}
 \end{theorem}

As mentioned in the introduction, Tcaciuc and Wallis proved in \cite{TW17} the following theorem:

 \begin{proposition}\cite[Proposition 2.11]{TW17}
Let $X$ be an infinite-dimensional complex Banach space and $T\in\B(X)$ a bounded operator such that $\sigma(T)$ is countable and $\sigma_p(T)=\emptyset$. Then for any nonzero $x\in X$ and any $\varepsilon>0$ there exists $F\in\B(X)$ with $\norm{F}<\veps$ and $\Range(F)=[x]$, and such that $T+F$ admits an invariant half-space.
\end{proposition}

When $X$ is reflexive, a companion result, \cite[Proposition 2.12]{TW17} allows one to (separately) control the kernel of the perturbation. Our main result in this section shows that for quasinilpotent operators we can control \emph{both} the range and the kernel at the same time, in a very strong way: with at most two exceptions, all perturbations by scalar multiples of a fixed rank-one operator have invariant half-spaces. Also note that this results holds in general Banach spaces. No reflexivity condition is needed.

\begin{theorem} \label{main2}
Let $X$ be a separable Banach space and $T\in\mathcal{B}(X)$ a quasinilpotent operator such that $\sigma_p(T)=\sigma_p(T^*)=\emptyset$.  Then for any rank one operator $F$, and any non-zero $\alpha\in\mathbb{C}$, with possibly two exceptions, $T+\alpha F$ has an invariant half-space.

\end{theorem}
\begin{proof}
It is easy to check that  $\sigma_p(T)=\emptyset$ if and only if $T$ has no non-trivial finite dimensional invariant subspaces, and that  $\sigma_p(T^*)=\emptyset$ if and only if $T$ has no non-trivial finite codimensional invariant subspaces. Therefore we can conclude from the hypotheses that any non-trivial invariant subspace of $T$ must be a half-space.

Let $F=e^*\otimes f$ be a rank-one operator, and consider the orbit $(T^nf)$. If there exists $k\in\N$ such that $e^*(T^nf)=0$ for all $n>k$, then $Y:=[T^nf]_{n>k}$ is an invariant subspace for $T$, contained in the kernel of $F$. Therefore $Y$ is a $T$-invariant half-space, and it is also invariant for $T+\alpha F$, for any $\alpha\in\C$.

There remains to consider the situation when $e^*(T^nf)\neq 0$ for infinitely many values of $n$. In this case it follows from the proof of Theorem \ref{main} that for all non-zero $\alpha\in\mathbb{C}$, with possibly one exception, $\sigma_p(T+\alpha F)$ is countably infinite. Moreover, $0$ is the only accumulation point for $\sigma_p(T+\alpha F)$. Denote by $\mathbb{C}_0$ the set of all these values $\alpha$; in other words, $\mathbb{C}_0$ does not contain $0$, and at most one more other value, depending on $F$.

For any $\alpha\in\mathbb{C}_0$,  define the set $S_{\alpha}$ as
$$
S_{\alpha}:=\{R(z)f : z\in\sigma_p(T+\alpha F)\setminus\{0\}\}\subseteq X.
$$

Note from Remark \ref{remark1} that $S_\alpha$ is a set of (linearly independent) eigenvectors corresponding to all distinct eigenvalues from   $\sigma_p(T+\alpha F)\setminus\{0\}$. For any $z\in\sigma_p(T+\alpha F)\setminus\{0\}$ we have that $e^*(R(z)f)=\alpha^{-1}$, hence $\norm{R(z)f}\geq (|\alpha|\norm{e^*})^{-1}$. That is, for any $\alpha\in\C_0$, $S_{\alpha}$ is bounded away from zero, therefore $0$ does not belong to the norm closure of $S_{\alpha}$.  Define the following sets:
\begin{eqnarray}
 \nonumber A &:=& \{\alpha\in\C_0 : S_\alpha\ \text{is not bounded}\} \\
 \nonumber B &:=& \{\alpha\in\C_0 : S_\alpha\ \text{is bounded and } \overline{S_\alpha}^{w}\  \text{is not } w \text{-compact}\} \\
 \nonumber C &:=& \{\alpha\in\C_0 : S_\alpha\ \text{is bounded and } \overline{S_\alpha}^{w}\  \text{is } w \text{-compact}\}.
\end{eqnarray}
Clearly $\C_0=A\cup B\cup C$, and the union is disjoint. We are going to show that for any $\alpha\in A\cup B$, $T+\alpha F$ has an invariant half-space, and that $|C|\leq 1$.

Let first $\alpha\in A$. Denote $\sigma_p(T+\alpha F)=(\lambda_n)_n$, and note that we have that $\lambda_n\to 0$. We are going to show that $S_\alpha$ contains a basic sequence. Since $S_\alpha$ is not bounded, by passing to a subsequence we may assume that $\norm{R(\lambda_n)f}\to\infty$. For any $n\in\N$, denote by $x_n:=R(\lambda_n)f/\norm{R(\lambda_n)f}$, and put $W_\alpha:=\{x_n : n\in\N\}$.

 The set $W_\alpha$ is bounded.  If  $\overline{W_\alpha}^{w}$ is not weakly-compact, we can apply Kadets-Pe{\l}czy\'{n}ski criterion (Theorem \ref{criterion}) to conclude that $W_\alpha$ contains a basic sequence. Therefore, by passing to a subsequence, we can assume that $(x_n)$ is a basic sequence in $X$. Then $Y:=[x_{2n}]$ is a half-space which is invariant for $T+\alpha F$.

  If $\overline{W_\alpha}^{w}$ is weakly compact, then it is weakly sequentially compact by the Eberlein-\v{S}mulian theorem, and by passing to a subsequence we can assume that $x_n\wto x\in X$. It is easy to see that
\begin{equation}\label{equ}
  Tx_n=\lambda_n x_n - \frac{1}{\norm{R(\lambda_n)f}}f.
\end{equation}

Since $Tx_n\tow Tx$, $\lambda_nx_n\tow 0$, and $\norm{R(\lambda_n)f}\tow\infty$, it follows from (\ref{equ}) that $Tx=0$. However $0$ is not an eigenvalue for $T$, so we must have $x=0$. Hence $0\in\overline{W_\alpha}^{w}$, and again by the Kadets-Pe{\l}czy\'{n}ski criterion we have that $W_\alpha$ contains a basic sequence, and we finish up as in the case when  $\overline{W_\alpha}^{w}$ is not weakly-compact.

When $\alpha\in B$, therefore $S_\alpha$ is bounded and $\overline{S_\alpha}^{w}$ is not weakly compact, we can again apply the Kadets-Pe{\l}czy\'{n}ski criterion to conclude that $S_\alpha$ contains a basic sequence, and again finish up as before. Therefore, we have shown that for $\alpha\in A\cup B$, $T+\alpha F$ has an invariant half-space. There remains to show that $|C|\leq 1$.

Assume towards a contradiction that there exist $\alpha\neq\beta$ in $C$.  Denote by $(\lambda_n)$ and by $(\mu_n)$ the eigenvalues in $\sigma_p(T+\alpha F)\setminus\{0\}$ and $\sigma_p(T+\beta F)\setminus\{0\}$, respectively, and note that both $(\lambda_n)$ and $(\mu_n)$ converge to $0$. For each $n\in\mathbb{N}$, we define $h_n:=R(\lambda_n)f$, and by $k_n:=R(\mu_n)f$. We have
 \begin{equation} \label{equ2}
 Th_n=\lambda_n h_n-f {\rm \hskip 2cm and \hskip 2cm  } Tk_n=\mu_n k_n-f.
 \end{equation}

 Since $\overline{S_\alpha}^{w}$ and $\overline{S_\beta}^{w}$ are weakly compact, we can assume, by passing to subsequences, that $h_n\wto h$ and $k_n\wto k$. Note that for any $n\in\N$ we have that
 \begin{equation*}
   e^*(h_n)=g(\lambda_n)=\alpha^{-1} {\rm \hskip 2cm and \hskip 2cm  } e^*(h_n)=g(\mu_n)=\beta^{-1}.
 \end{equation*}
Therefore, $e^*(h)=\alpha^{-1}$ and $e^*(k)=\beta^{-1}$, and since $\alpha\neq\beta$, it follows that $h\neq k$. Taking weak limits in (\ref{equ2}), and taking into account that $\lambda_n\to 0$ and $\mu_n\to 0$, we get that $Th=-f$ and $Tk=-f$. Therefore $T(h-k)=0$, and since $h-k\neq 0$ it follows that $0$  is an eigenvalue for $T$, which is a contradiction since $\sigma_p(T)=\emptyset$. It follows that $|C|\leq 1$, and this completes the proof.
\end{proof}

While not explicitly stated in the previous theorem, note that, in particular, we can obtain rank-one perturbations of arbitrarily small norms that have invariant half-spaces. Indeed, since -- for a fixed rank-one $F$ -- almost all perturbations $T+\alpha F$ have invariant half-spaces, for any given $\veps>0$ we can take a \enquote{good} $\alpha<\veps/\norm{F}$. We summarize this in the following corollary:

\begin{corollary}
Let $X$ be a separable Banach space and $T\in\mathcal{B}(X)$ a quasinilpotent operator such that $\sigma_p(T)=\sigma_p(T^*)=\emptyset$. Then for any non-zero $f\in X$, $e^{*}\in X^*$, and $\veps>0$, we can find  rank-one $F\in\mathcal{B}(X)$ with $\Range(F)=[f]$, $\ker{F}=\ker{e^*}$, and $\norm{F}<\veps$ such that $T+F$ has an invariant half-space.
\end{corollary}

In the Hilbert space setting we get more specific information about the structure of a quasinilpotent operator.

\begin{corollary} \label{hilbert}
Let $\mathcal{H}$ be a separable Hilbert  space and $T\in\mathcal{B}(\mathcal{H})$ a quasinilpotent operator such that $\sigma_p(T)=\sigma_p(T^*)=\emptyset$.  Then for any rank one operator $F$, and any non-zero $\alpha\in\mathbb{C}$, with possibly two exceptions, there exists an orthogonal projection of infinite rank and co-rank such that $P^{\perp}TP=\alpha P^{\perp}FP$.
\end{corollary}

\begin{proof}
  Fix a rank one operator $F\in\mathcal{B}(\mathcal{H})$. From Theorem \ref{main2} we have that for all non-zero $\alpha\in\mathbb{C}$, with possibly two exceptions, $T-\alpha F$ has an invariant half-space. Fix such an $\alpha\in\mathbb{C}$, let $Y$ be an invariant half-space for $T-\alpha F$, and let $P\in\mathcal{B}(\mathcal{H})$ be the orthogonal projection onto $Y$ (which clearly has infinite rank and co-rank). Since $Y$ is invariant for $T-\alpha F$ it is easy to see that $P^{\perp}(T-\alpha F)P=0$, and the conclusion follows.
\end{proof}

\section{Open questions}

In light of Theorem \ref{main}, the behaviour of the spectrum of a quasinilpotent operator under rank one perturbations is related to a solution for the invariant Subspace Problem for quasinilpotent operators. This suggests a natural and very important open question in this direction.

{\bf Question 1:} Given $T\in\mathcal{B}(X)$ a quasinilpotent operator acting on an infinite dimensional, separable, complex Banach space, can we find a rank-one operator $F$ such that $T+F$ is quasinilpotent?
\vskip .3cm
Note that a positive answer to this question is not sufficient to conclude a positive solution to the Invariant Subspace Problem for quasinilpotent operators by applying Theorem \ref{main}. Indeed, from $iii)$ in Theorem \ref{main} we would need one more perturbation by a scalar multiple of the same rank one operator that is also quasinilpotent. On the other hand, a negative answer to the \emph{Question 1} will provide a counterexample to ISP for quasinilpotent operators. As we mentioned before, Read \cite{R97} already provided such a counterexample on $l_1$, therefore it would be a natural starting point to examine \emph{Question 1} for Read's operator.

The requirement that $F$ has rank one is very important, as we can always find $F$ of rank two such that $T+F$ is quasinilpotent. Indeed, let $N$ be a rank one nilpotent operator (thus $N^2=0$) and put $S:=(I-N)T(I+N)$. Then $S$ is similar to $T$, therefore $S$ is also quasinilpotent, and an easy calculation shows that $T-S$ has rank at most two.

In Section 2 we constructed an example of a quasinilpotent operator $T\in\mathcal{B}(l_2)$, and a rank one operator $F$ such that $T+F$ is quasinilpotent, but $T+\alpha F$ is not, for all $\alpha\neq 0$, $\alpha\neq 1$. Therefore $F$ does not satisfy $iii)$ in  Theorem \ref{main}, but of course there are other rank one operators that do, as $T$ in that example has plenty of invariant subspaces. Whether a positive solution to \emph{Question 1} already implies a positive solution to ISP for quasinilpotent operators is another important open question:

{\bf Question 2:} If $T\in\mathcal{B}(X)$ is a quasinilpotent operator with the property that there exists $F\in\mathcal{B}(X)$  rank one such that $T+F$ is quasinilpotent, does $T$ have an invariant subspace?
\vskip .3cm
Note that if \emph{Question 1} has  a positive solution for Read's operator, then Read's operator provides a negative answer to \emph{Question 2}.

\vskip .3cm

\emph{Acknowledgments.} It is my pleasure to acknowledge the helpful conversations and feedback provided to me by Heydar Radjavi. This research was supported in part by the Natural Sciences and Engineering Research Council of Canada.

\end{document}